\documentclass[11pt]{amsart}
\usepackage{amsmath,amsthm,amssymb, enumerate}
            \newcommand{\marginalnote}[1]{}
\usepackage[english]{babel}
\usepackage{hyperref}
\usepackage{marginnote}
\usepackage[all]{xy}
\usepackage{graphicx}
\usepackage{mathrsfs}
\usepackage{epsfig}
\usepackage{color}	
\usepackage{ulem}	
\usepackage{caption} 
\usepackage{mathtools}
\theoremstyle{plain}

\newtheorem{Thm}{Theorem}[section]

\newtheorem{Lem}[Thm]{Lemma}

\newtheorem*{Thm*}{Theorem}
\theoremstyle{remark}

\theoremstyle{definition}

\newtheorem{Ex}[Thm]{Example}

\newtheorem{Que}[Thm]{Question}
\usepackage[left=0.75in,right=0.75in,top=1.5in,bottom=1.5in]{geometry}

\usepackage{etoolbox}

\makeatletter
\patchcmd{\epigraph}{\@epitext{#1}}{\itshape\@epitext{#1}}{}{}
\makeatother

\begin{document}
\title{Three-dimensional maps and subgroup growth}
\author{R\'emi Bottinelli, Laura Ciobanu, Alexander Kolpakov}
\date{\today}

\begin{abstract}
In this paper we derive a generating series for the number of cellular complexes known as pavings or three-dimensional maps, on $n$ darts, thus solving an analogue of Tutte's problem in dimension three. 

The generating series we derive also counts free subgroups of index $n$ in $\Delta^+ = \mathbb{Z}_2*\mathbb{Z}_2*\mathbb{Z}_2$ via a simple bijection between pavings and finite index subgroups which can be deduced from the action of $\Delta^+$ on the cosets of a given subgroup. We then show that this generating series is non-holonomic. Furthermore, we provide and study the generating series for isomorphism classes of pavings, which correspond to conjugacy classes of free subgroups of finite index in $\Delta^+$.  

Computational experiments performed with software designed by the authors provide some statistics about the topology and combinatorics of pavings on $n\leq 16$ darts. 

\medskip


\medskip

\end{abstract}

\maketitle

\section{Introduction}

In this note we explore the correspondence between the number of rooted three-dimensional maps, or pavings, on $n$ darts, as introduced in \cite{AK, Lienhardt, Spehner}, and free subgroups of given index $n$ in the free product $\Delta^+ = \mathbb{Z}_2* \mathbb{Z}_2* \mathbb{Z}_2$, in order to obtain generating series, new formulas and asymptotics for these objects. For any surface or higher-dimensional manifold that has been triangulated or otherwise subdivided into cells (not necessarily simplices), combinatorial maps are a way of recording the neighbouring relations between cells (vertices, edges, faces, etc), such as incidence or adjacency. The number of \textit{darts} (defined in Sections~\ref{s:2-d-maps} and \ref{s:3-d-maps}), which are essentially edges or half-edges, is for us the key parameter in quantifying the number of maps, and can be seen as an ``elementary particle'' from which the combinatorial objects in this paper are assembled. 

There is a natural way to associate with every free subgroup of index $n$ in $\Delta^+$ a paving on $n$ darts, and we give new quantitative information, as well as examples with concrete computations, for both kinds of objects, the geometric ones and the algebraic ones. We also count the conjugacy classes of free subgroups of index $n$ in $\Delta^+$, and investigate the link between these and isomorphism classes of pavings. 

While similar connections between free subgroups (and their conjugacy classes) of finite index in certain Fuchsian triangle groups and two-dimensional maps have been previously exploited by a number of authors (\cite{Breda-Mednykh-Nedela, Jones-Singerman, M, MN1, MN2, Petitot-Vidal, Stothers-modular, Stothers, Vidal}), relatively little has been done for maps in $3$ dimensions; this paper is a step towards developing the theory and computation in higher dimensions. In particular, this paper provides a solution to the analogue of Tutte's problem (enumeration of isomorphisms classes of maps and hypermaps) in dimension $3$. 

Also, three-dimensional maps, or pavings, closely resemble the ``edge coloured graphs'' (as des\-cri\-bed by Gurau in \cite{Gurau}) used in order to study random tensors and associated tensor integrals, which can be viewed as a generalisation of matrix integrals related to counting maps and hypermaps in dimension $2$. Therefore, pavings can be viewed as a first step in quantifying Gurau's approach.

General subgroup growth is the subject of the book \cite{LS}, and further information on subgroup growth in free products of cyclic groups can be found in \cite{BPR, Mul91, Mul96, MSP1, MSP2, MSP3, Stothers}. There, the general theory of subgroup structure in free products of (finite and infinite) cyclic groups is enhanced by using the methods of representation theory, analytic number theory and probability theory, among other tools.

The novelty of our contribution is in the methods we use, which have not been employed for counting pavings before and which are particularly suitable for practical computations, as well as in the qualitative information about the generating series we obtain, such as the fact that they are non-holonomic.
We use the species theory initiated by Joyal \cite{Joyal} (c.f. the monographs \cite{Bergeron-et-al, Flajolet-Sedgewick}) as our main computational tool, which allows us to derive the exponential generating series for the number of rooted pavings in Theorem \ref{thm:rooted-pavings} (or free subgroups of finite index in Theorem \ref{thm:subgroups}) and the number of isomorphism classes of connected pavings in Theorem \ref{thm:pavings} (or conjugacy classes of said subgroups in Theorem \ref{thm:conjugacy-classes}) in a relatively simple form suitable for routine calculation and computer experiments. 
We are able to associate the generating series for the number of rooted pavings with solutions of the classical Riccati equations, which shows they are non-holonomic by a result of \cite{Klazar}. Further connections between map enumeration and the Riccati equation were established in \cite{AB}. 

Throughout the paper we give several concrete and illustrative examples, and in Section \ref{sec:stats} we provide some statistical information about pavings on $n\leq 16$ darts using a computer program \textrm{Nem} \cite{Nem} created for the purpose of their enumeration and classification. 

\begin{small}
\section*{Acknowledgements}
The authors gratefully acknowledge the support received from the University of Neuch\^atel Overhead grant no.~12.8/U.01851. R.B. and A.K. were also supported by FN~PP00P2-144681/1, and L.C. was supported by  FN~PP00P2-170560 projects of the Swiss National Science Foundation. The authors greatly appreciate the fruitful discussions with the On-line Encyclopaedia of Integer Sequences  -- OEIS  community, who helped identify integer sequences in this manuscript and improve its exposition. The computations were performed using the ``Cervino'' computational cluster of the Computer Science Department at the University of Neuch\^atel. 
\end{small}

\section{Preliminaries}

\subsection{Two-dimensional maps}\label{s:2-d-maps}

A two-dimensional oriented combinatorial map or, simply, a combinatorial map, is a triple $H = \langle D; \alpha, \sigma \rangle$, where $D = \{ 1, 2, \dots, n \}$ is a finite set of $n\geq 0$ \textit{darts} (to be defined below), $\alpha$, $\sigma \in \mathfrak{S}_n$ are permutations of $D$, and $\alpha$ is an involution. A map $H$ is \textit{connected} if the group $G_H = \langle \alpha, \sigma \rangle$ acts transitively on $D$.  

Any combinatorial map has a topological realisation $\Gamma_H$ as a disjoint union of connected graphs, each embedded into a connected orientable surface. In order to construct $\Gamma_H$, one may proceed as follows. Let $\phi = \sigma^{-1} \alpha$, and for each cycle of $\phi$ consider a polygon, called a \textit{face} of $\Gamma_H$, whose edges are oriented anticlockwise. 
Two edges $i$ and $j$ of the newly produced faces are identified in accordance with the transpositions of $\alpha$, that is, if $\alpha(i)=j$ then $i$ is identified with $j$, and each new edge becomes the union of the now two half-edges or \textit{darts} $i$ and $j$, pointing in opposite directions (and each towards a vertex). This ensures that the resulting topological space $\Gamma_H$ is orientable. 
The ordered sequence of darts pointing towards a vertex of $\Gamma_H$ is now described by a suitable cycle of $\sigma$. Thus the vertices of $\Gamma_H$ correspond to the disjoint cycles of $\sigma$.

By construction, the topological space that we obtain after performing the procedure above is an oriented surface without boundary, which is connected if $G_H$ acts transitively on $D$. However, we do not always assume connectivity/transitivity.    

The above argument establishes a bijection between combinatorial maps and topological maps, i.e. graphs embedded into orientable (possible disconnected) surfaces, where for each connected component $\langle \Sigma_g; \Gamma, \iota \rangle$ with $\Sigma_g$ a genus $g$ surface, and $\Gamma$ embedded in $\Sigma_g$ via the map $\iota$, the complement $\Sigma_g\setminus \iota(\Gamma)$ is a union of topological discs. Each edge of such a $\Gamma$ is split into a pair of labelled half-edges pointing in opposite directions. The darts $D$ are exactly those oriented half-edges.

The permutations $\alpha$, $\sigma$ and $\phi = \sigma^{-1} \alpha$ defining $H$ can be read off the labelled topological map $\Gamma_H$ as follows:
\begin{itemize}
\item[1)] the cycles of $\alpha$ correspond to the darts forming entire edges of $\Gamma_H$,
\item[2)] the cycles of $\sigma$ correspond to the sequences of darts around vertices read in an anticlockwise direction,
\item[3)] the cycles of $\phi$ correspond to the sequences of darts obtained by moving around faces in an anticlockwise direction.
\end{itemize}

Two combinatorial maps $H_1 = \langle D; \alpha_1, \sigma_1 \rangle$ and $H_2 = \langle D; \alpha_2, \sigma_2 \rangle$ are isomorphic if there exists $\pi \in \mathfrak{S}_n$ such that $\alpha_1 = \pi^{-1}\, \alpha_2\, \pi$ and $\sigma_1 = \pi^{-1}\,\sigma_2\, \pi$, which for the associated topological maps translates into the existence of an orientation-preserving homeomorphism between $\Gamma_{H_1}$ and $\Gamma_{H_2}$ that respects dart adjacencies. 

For any permutations $\pi_i \in \mathfrak{S}_n$, $i = 1, \dots, l$, let $\zeta(\pi_1, \dots, \pi_l)$ be the number of orbits of the group $\langle \pi_1, \dots, \pi_l \rangle$ acting on $D = \{1, 2, \dots, n\}$. Then the connected components of $H = \langle D; \alpha, \sigma \rangle$  are represented by the orbits of $\langle \alpha, \sigma \rangle$, the faces of $H$ are the orbits of $\langle \sigma^{-1} \alpha \rangle$, and its edges and vertices are the orbits of $\langle \alpha \rangle$ and $\langle \sigma \rangle$, respectively. Thus the Euler characteristic of $H$ can be defined as $\chi(H) = \zeta(\sigma^{-1} \alpha) - \zeta(\alpha) + \zeta(\sigma)$.   

\subsection{Three-dimensional maps}\label{s:3-d-maps}

A three-dimensional oriented combinatorial map or, simply, a (combinatorial) \textit{paving}, is a quadruple $P = \langle D; \alpha, \sigma, \varphi \rangle$, where $D$ is an $n$-element set ($n\geq 0$) and $\alpha, \sigma, \varphi \in \mathfrak{S}_n$ are permutations of $D$ such that $H = \langle D; \alpha, \sigma \rangle$ is a map (not necessarily connected), and:
\begin{itemize}
\item[(I-1)] the product $\alpha \varphi$ is an involution,
\item[(I-2)] the product $\varphi \sigma^{-1}$ is an involution,
\item[(FP)] neither of the above involutions has fixed points.
\end{itemize}

A paving $P$ is connected if $G_P = \langle \alpha, \sigma, \varphi \rangle$ acts transitively on $D$. Given a paving $P = \langle D; \alpha, \sigma, \varphi \rangle$, the map $H = \langle D; \alpha, \sigma \rangle$ is called the \textit{underlying map} of $P$. 

We may also think of $P$ as a quadruple $P = \langle D; \alpha, \beta, \gamma \rangle$, where $D$ is an $n$-element set ($n\geq 0$) of darts and $\alpha, \beta, \gamma \in \mathfrak{S}_n$ are involutions without fixed points. In this case it is easy to see that letting $\varphi = \alpha \beta$ and $\sigma = \gamma \alpha \beta $ produces the initial definition. 

As in the case of two-dimensional maps, a combinatorial paving $P$ has a topological realisation $M_P$ obtained as follows. Let $H_P = \langle D; \alpha, \sigma \rangle$ be the underlying map for a paving $P = \langle D; \alpha, \sigma, \varphi \rangle$, and let us realise each connected component of $H$ as a topological map, i.e. view $H_P$ as a collection of surfaces $\Sigma^i$ with embedded graphs $\Gamma^i$, $i=1, 2, \dots, m$, having labelled half-edges as described in Section~\ref{s:2-d-maps}. Each surface $\Sigma^i$ represents the boundary of a handle-body $B^i$, and then the handle-bodies $B^i$ become identified along their boundaries in order to produce a labelled oriented cellular complex representing $P$ topologically. Indeed, the faces of $\Sigma^i$'s defined by the permutation $\sigma^{-1} \alpha$ are identified in accordance with the permutation $\varphi$, and the conditions I-1, I-2, FP ensure that one face cannot be identified to multiple disjoint counterparts  (implied by conditions I-1 and I-2), and edges or faces cannot bend onto themselves (implied by condition FP). Also, conditions I-1 and I-2 ensure that the faces of two disjoint handle-bodies come together with coherent orientations, thus resulting in an orientable topological space $M_P$.  However, we'd like to note that $M_P$ is not always a three-dimensional manifold. Such an example can be delivered by Thurston's figure-eight glueing from \cite[Ch.~1, p.~4]{Thurston-notes}, described in Example \ref{ExThurston} (c.f. \cite[\S 10.3]{Ratcliffe} for a more detailed description). 

A paving $P$ is \textit{rooted} if one of its darts is singled out as a \textit{root} dart. In the sequel, we shall always assume that the root dart has label $1$. 

The definitions of isomorphism for combinatorial and topological pavings are absolutely analogous to those for combinatorial and topological maps.

\begin{Ex}\label{ExThurston}
Let $D = \{ 1, 2, \dots, 12 \} \cup \{-1, -2, \ldots, -12  \}$ be a set. Let $\alpha$, $\sigma$ and $\varphi$ be the following permutations of $D$:
\begin{equation*}
\alpha = (1,2)(3,4)(5,6)(7,8)(9,10)(11,12)(-1,-2)(-3,-4)(-5,-6)(-7,-8)(-9,-10)(-11,-12),
\end{equation*}
\begin{equation*}
\sigma = (1,5,3)(2,9,8)(4,11,10)(6,7,12)(-1,-10,-7)(-2,-6,-4)(-3,-12,-9)(-5,-8,-11),
\end{equation*}
\begin{equation*}
\varphi = (1,-11,3,-2,12,-4)(2,-3,11,-1,4,-12)(5,-5,8,-9,9,-8)(6,-7,10,-10,7,-6).
\end{equation*}

\begin{figure}[ht]
\includegraphics*[scale=0.5]{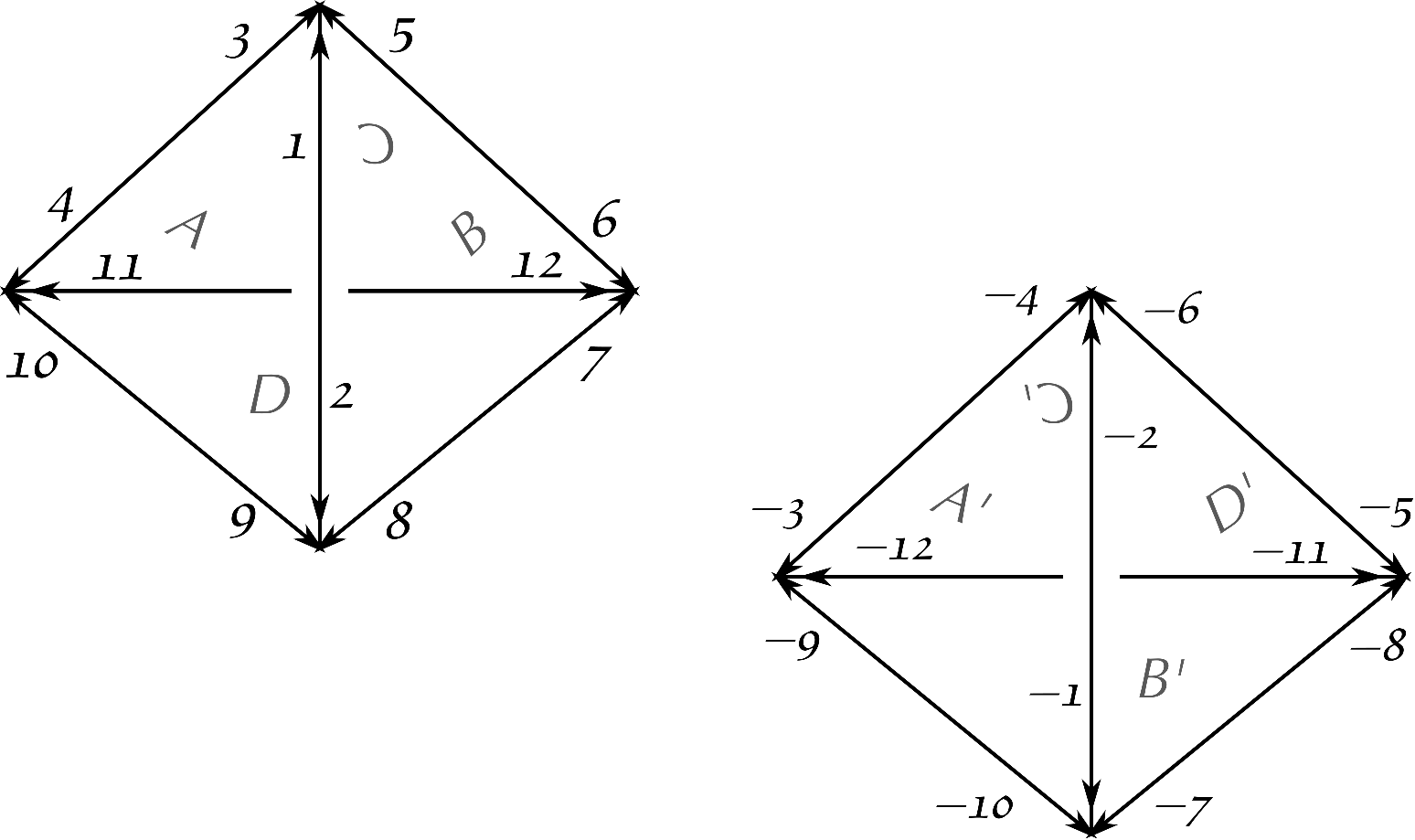}
\caption{Two tetrahedra used in Thurston's figure-eight glueing. Here, they do not need to be geometrically realisable. Each face labelled $X$ glued to the respective face labelled $X'$, while the darts match as described by $\phi$.}\label{fig:Th-map}
\end{figure}

Consider the paving $P= \langle D; \alpha, \sigma, \varphi \rangle$, whose underlying map $H_P = \langle D; \alpha, \sigma \rangle$ consists of the two tetrahedra depicted in Figure~\ref{fig:Th-map}. The face glueing is described by $\phi = \alpha \varphi$, while $\varphi$ describes the equivalence classes of the edges (or the so-called edge cycles, c.f. \cite[\S 10.1]{Ratcliffe}) under the glueing. We have 
\begin{equation*}
\phi = \left(  \begin{array}{cccc}
\overbrace{1\,\,\,\,\,\,\,\,\,\, 4\,\,\,\,\,\,\,\,\, 9}^{A}& \overbrace{7\,\,\,\,\,\,\,\,\,\, 5\,\,\,\,\,\,\,\,\, 2}^{B}& \overbrace{3\,\,\,\,\,\,\,\,\,\, 11\,\,\,\,\,\,\,\,\,\, 6}^{C}& \overbrace{8\,\,\,\,\,\,\,\,\,   10\,\,\,\,\,\,\,\,\,  12}^{D} \\
\underbrace{-3 -2 -10}_{A'}& \underbrace{-9 -7 -11}_{B'}& \underbrace{-12\,\, -4\,\, -5}_{C'}& \underbrace{-6\,\,\, -8\,\,\, -1}_{D'}
\end{array} \right).
\end{equation*}

After performing the necessary identifications, we obtain a cellular space with Euler characteristic $+1$, which has two $3$-cells, four $2$-cells, two $1$-cells, and a single $0$-cell. Therefore one does not obtain a manifold, since $3$-manifolds have zero Euler characteristic \cite[Theorem 4.3]{FM}.  

Indeed, the link of $P$ can be computed as a map $L_P = \langle D;  \varphi^{-1} \sigma, \sigma^{-1} \rangle$, c.f. \cite[Proposition 4.1]{Spehner}. We obtain that its Euler characteristic $\chi(L_P) = \xi(\sigma \varphi^{-1} \sigma) - \xi(\varphi^{-1}\sigma) + \xi(\sigma^{-1}) = 4 - 12 + 8 = 0$, which means that $L_P$ is indeed a torus. 
\end{Ex}

The approach to pavings described above is largely due to Spehner, c.f. \cite{Spehner}. Another, effectively dual, approach is due to Arqu\`{e}s and Koch \cite{AK}, and these two approaches to pavings are shown to be equivalent by Lienhardt in \cite{Lienhardt}.  

Arqu\`{e}s and Koch's approach is as follows. Let $P = \langle D; \alpha, \sigma, \varphi \rangle$ be a combinatorial paving. Then we assemble an oriented cellular complex $M_P$ in such a way that the underlying map $H_P = \langle D; \alpha, \sigma \rangle$ produces (possibly disjoint) links of vertices in $M_P$. Each link is a map whose edges are intersections of the two-dimensional angular segments (or, simply, labelled corners of its two-faces \cite[D\'{e}finition 2.2, 1) \& 2)]{AK}, c.f. discussion in \cite[p.~71]{Lienhardt}) representing the darts $D$ and emanating from each vertex, with the respective link surface. In this case, the latter should be thought of as the boundary of a sufficiently small neighbourhood of said vertex. Then $\varphi$  brings angular segments belonging to the same two-cell of $M_P$ together, which finalises the construction. One may also take $H^*_P = \langle D; \varphi^{-1} \sigma, \sigma^{-1} \rangle$ as the underlying map and perform Spehner's construction as previously described (with the 
r\^{o}les of the associated permutations described in \cite[Proposition 4.1]{Spehner}, passing to the dual map). Finally, $M_P$ is a topological presentation for $P$.

In the rest of the paper we follow Spehner's approach, in which the vertex links $L_P$ of a paving $P = \langle D; \alpha, \sigma, \varphi \rangle$ are described by the underlying map of its dual $P^*$. Namely, we have $L_P = H_{P^*} = \langle D; \varphi^{-1} \sigma, \sigma^{-1} \rangle$, c.f. \cite[Proposition 4.1]{Spehner}. Thus, given a paving $P$, computing its links $L_P$ is straightforward.

Every paving $P$ gives rise to an oriented pseudo-manifold $M_P$ as its topological presentation, as follows from its definition. A paving $P$ with $M_P$ an oriented three-manifold has all links homeomorphic to the sphere $\mathbb{S}^2$. If $L_P$ contains some connected components of genus higher than zero, we excise the non-spherical links from $M_P$ and obtain a manifold with boundary. 

For a paving $P = \langle D; \alpha, \sigma, \varphi \rangle$, let the number of connected components of its underlying map $H_P = \langle D; \alpha, \sigma \rangle$ be $f_3 := \zeta(\alpha, \sigma)$, which is also the number of connected three-dimensional handlebodies constituting $M_P$, or the number of ``pieces,'' as described in \cite[Definition~1.5]{Spehner}. The number of two-dimensional faces of $P$ is $f_2 := \zeta(\sigma^{-1}\alpha, \varphi^{-1}\sigma)$, the number of edges is $f_1 := \zeta(\alpha, \varphi)$ and that of vertices is $f_0 := \zeta(\sigma, \varphi)$.  

The $\mathrm{f}$-vector of $P$ is $\mathrm{f}(P) := (f_0, f_1, f_2, f_3)$. The complexity of $P$ equals $\mathrm{c}(P) = f_3 - f_2 + f_1 - f_0$. In general, this quantity does not coincide with the Euler characteristic of $P$, unless the underlying map $H$ is planar (i.e. all the connected components of $H$ are spheres). 

\subsection{Formal power series}\label{s:species-series} 

Here we follow \cite{CK-2017}. A \textit{hypergeometric sequence} $(c_k)_{k\geq 0}$ has $c_0=1$ and enjoys the property that the ratio of its any two consecutive terms is a rational function in $k$, i.e. there exist monic polynomials $P(k)$ and $Q(k)$ such that $$\frac{c_{k+1}}{c_k}=\frac{P(k)}{Q(k)}.$$

Moreover, if $P$ and $Q$ are factored as $$\frac{P(k)}{Q(k)}=\frac{(k+a_1)(k+a_2)\dots(k+a_p)}{(k+b_1)(k+b_2)\dots(k+b_q)(k+1)},$$
then we use the notation 
$$ {}_{p}F_{q}\left[ \begin{array}{c} a_1 \dots a_p\\ b_1 \dots b_q 
\end{array}; z \right]$$ for the formal series $F(z)=\sum_{k\geq 0} c_kz^k$, c.f. \cite[\S 3.2]{Zeilberger-et-al}. Here, the factor $(k+1)$ belongs to the denominator for historical reasons. Such a hypergeometric series satisfies the differential equation
\begin{equation}\label{HGSdiffeq}
\Big(\vartheta(\vartheta+b_{1}-1)\cdots(\vartheta+b_{q}-1)-z(\vartheta+a_{1})%
\cdots(\vartheta+a_{p})\Big)\,\, {}_{p}F_{q}(z) = 0,
\end{equation}
where $\vartheta=z\frac{d}{dz}$, c.f. \cite[\S 16.8(ii)]{DLMF}. Among numerous differential equations related to (\ref{HGSdiffeq}) is the \textit{classical Riccati equation}, which plays an important r\^{o}le later on. It is a first order non-linear equation with variable coefficients $f_i(x)$, of the form 
\begin{equation}\label{Riccati2}
\frac{\mathrm{d}y}{\mathrm{d}x}= f_1(x) + f_2(x) y + f_3(x) y^2.
\end{equation}

The \textit{Pocchammer symbol} is connected to hypergeometric series and defined as $$(a)_n=a(a+1)\dots (a+n-1).$$ As $n \rightarrow \infty$, it has the following asymptotic expansion
\begin{equation}\label{Pocchammer}
(a)_n \propto \frac{\sqrt{2\pi}}{\Gamma(a)}\, e^{-n}\, n^{a + n - \frac{1}{2}},
\end{equation}
where $\Gamma(a)$ is the Gamma function of $a$, defined as $\Gamma(a)=(a-1)!$ for $a$ a positive integer, and $\Gamma(a)= \int_0^{\infty} x^{a-1}e^{-x} dx$ for all the non-integer real positive numbers.

A formal power series $y=f(x)$ is called \textit{D-finite}, or \textit{differentiably finite}, or \textit{holonomic}, if there exist polynomials $p_0, \dots, p_m$ (not all zero) such that $p_m(x)y^{(m)}+ \dots +p_0(x)y=0$, where $y^{(m)}$ denotes the $m$-th derivative of $y$ with respect to $x$. All algebraic power series are holonomic, but not vice versa, c.f. \cite[Appendix B.4]{Flajolet-Sedgewick}.

Finally, we recall that the \textit{Hadamard product} of two formal single-variable series $A(z)=\sum_{n\geq 0} a_n \frac{z^n}{n!}$ and $B(z)=\sum_{n\geq 0} b_n \frac{z^n}{n!}$ is denoted $(A \odot B)(z)$ and given by $(A \odot B)(z):=\sum_{n\geq 0} a_n b_n \frac{z^n}{n!}$.

Let $\lambda = (n_1, \dots, n_m)$ be a partition of a natural number $n\geq 0$, i.e. $n = \sum_{i\geq 1}\, i n_i$. We write $\lambda \vdash n$ and define $\lambda ! := 1^{n_1} n_1! 2^{n_2} n_2! \dots m^{n_m} n_m!$. Let $\mathbf{z}^{\lambda}: = z_1^{n_1} z_2^{n_2} \dots z_m^{n_m}$ for some collection of variables $z_1$, $z_2$, $\dots, z_m$. Then for two multi-variable series $A(\mathbf{z}) = \sum_{n\geq 0} \sum_{\lambda \vdash n} a_\lambda \frac{\mathbf{z}^\lambda}{\lambda !}$ and $B(\mathbf{z}) = \sum_{n\geq 0} \sum_{\lambda \vdash n} b_\lambda \frac{\mathbf{z}^\lambda}{\lambda !}$ the Hadamard product is $(A\odot B)(\mathbf{z}) := \sum_{n\geq 0} \sum_{\lambda \vdash n} a_\lambda b_\lambda \frac{\mathbf{z}^\lambda}{\lambda!}$.

Also, for a multiple Hadamard product of a series $A(\mathbf{z})$ with itself, i.e. $B(\mathbf{z}) = (A \odot \dots \odot A) (\mathbf{z})$, we shall write $B(\mathbf{z}) = A^{\odot n}(\mathbf{z})$, with a suitable $n\geq 0$.

\subsection{Species theory}\label{s:species-theory}

Species theory (th\'{e}orie des esp\`{e}ces), initially due to A. Joyal \cite{Joyal}, is a powerful way to describe and count labelled discrete structures. Since it requires a lengthy and formal setup, we give here only the basic ideas and refer the reader to \cite{Bergeron-et-al, Flajolet-Sedgewick} for further details.

A \textit{species of structures} is a rule (or functor) $F$ which produces

\begin{itemize}
\item[i)] for each finite set $U$ (of labels), a finite set $F[U]$ of structures on $U$,
\item[ii)] for each bijection $\sigma: U\rightarrow V$, a function $F[\sigma]: F[U] \rightarrow F[V]$.
\end{itemize} 
 
The functions $F[\sigma]$ should further satisfy the following functorial properties:

\begin{itemize}
\item[i)] for all bijections $\sigma:U \rightarrow V$ and $\tau:V \rightarrow W$, 
$F[\tau \circ \sigma] = F[\tau]\circ F[\sigma]$,
\item[ii)] for the identity map $Id_U : U \rightarrow U$, $F[Id_U] = Id_{F[U]}$.
\end{itemize}

Let $[n] = \{1,2,\dots,n\}$ be an $n$-element set, and assume that $[0] = \emptyset$. A species $F$ of \textit{labelled structures} has exponential generating function $F(z) = \sum_{n \geq 0} \mathrm{card}\,F[n] \frac{z^n}{n!}$, where $\mathrm{card}\,F[n]$ denotes the size of $F[n]$. 

For a species of \textit{unlabelled structures} (i.e. structures up to isomorphism) we write $\widetilde{F}$, and its generating function is a specialisation of the cycle index series, in the sense that $\widetilde{F}(z)=\mathcal{Z}_F(z,z^2, z^3 \dots)$, where the \textit{cycle index series} (see \cite[\S 1.2.3]{Bergeron-et-al}) is defined as: $$ \mathcal{Z}_F(z_1,z_2, \dots) = \sum_{n\geq 0} \frac{1}{n!} \sum_{\sigma \in \mathfrak{S}_n} \mathrm{card}\, Fix(F[\sigma])\, \mathbf{z}^\sigma.$$ Here $Fix(F[\sigma])$ is the set of elements of $F[n]$ having $F[\sigma]$ as automorphism, and $\mathbf{z}^\sigma = z_1^{c_1} z_2^{c_2} \dots z_m^{c_m}$ if the cycle type of $\sigma$ is $c(\sigma) = (c_1, c_2, \dots, c_m)$ (i.e. $c_k$ is the number of cycles of length $k$ in the decomposition of $\sigma$ into disjoint cycles). 

\begin{Ex}
This example illustrates the difference between the exponential generating function $S(z)$ for a species of labelled structures and the generating function $\widetilde{S}(z)$ for the corresponding species of unlabelled structures in the setting of permutations. 
Let $[n] = \{1, 2, \dots, n\}$ be a finite set, $S[n]$ the species of all permutations $Sym(n)$ of $n$ distinct numbers (which are considered labelled structures), and $\widetilde{S}[n]$ the species of unlabelled structures. Then $\widetilde{S}[n]$ consists of all conjugacy classes of permutations in $Sym(n)$, and classical counting arguments give:
\begin{equation*}
S(z) = \sum_{n \geq 0} \mathrm{card}\, \mathcal{S}[n] \cdot \frac{z^n}{n!} = \sum_{n\geq 0} n! \cdot \frac{z^n}{n!}  = \frac{1}{1-z},
\end{equation*}
while 
\begin{equation*}
\widetilde{S}(z) = \sum_{n \geq 0} p(n)\, z^n = \prod_{n\geq 1} \frac{1}{1-z^n},
\end{equation*}
where $p(n)$ is the number of unordered partitions of $n \geq 0$\footnote{C.f. sequence A000041 in the OEIS \cite{OEIS}.}, with $p(0) = 1$, c.f. \cite[Exemple 9]{Bergeron-et-al}.
\end{Ex}

\section{Maps and subgroups}\label{s:maps-groups}

Let $\mathcal{P}(n)$ be the set of connected pavings on $n$ darts, and let $\mathcal{P}_r(n)$ be the set of connected rooted pavings on $n$ darts.
We will assume that if pavings are rooted they have root $1$. Let $P = \langle D; \alpha, \beta, \gamma \rangle$ be a rooted paving from $\mathcal{P}_r(n)$. Then there is an epimorphism $\psi$ from $\Delta^+ = \mathbb{Z}_2*\mathbb{Z}_2*\mathbb{Z}_2 \cong \langle a | a^2 = \varepsilon \rangle * \langle b | b^2 = \varepsilon \rangle * \langle c | c^2 = \varepsilon \rangle$ onto the group $G_P = \langle \alpha, \beta, \gamma \rangle \subset \mathfrak{S}_n$ given by $\psi: a \mapsto \alpha, b \mapsto \beta, c \mapsto \gamma$. Moreover, $\Delta^+$ acts transitively on $D$ via this epimorphism, since the action of $G_P$ is transitive. By taking $\Gamma := Stab(1)$ with respect to this action, we observe that the action of $\Delta^+$ on $D$ is isomorphic to the action of $\Delta^+$ on the set of cosets $\Delta^+\diagdown \Gamma$. 

If we consider the isomorphism class of $P$ or, equivalently, consider $P \in \mathcal{P}_r(n)$ as a representative from $\mathcal{P}(n)$,  a change of root in $P$ from $1$ to $i$ corresponds to conjugation of $\Gamma$ by an element $w \in \Delta^+$ such that $\omega = \psi(w)$ has the property $\omega(1) = i$. 

By an argument analogous to that of \cite[Lemmas 3.1-3.2]{CK-2017}, the following hold.

\begin{Lem}\label{lemma:rooted-pavings-subgroups}
There exists a bijection between the set $\mathcal{P}_r(n)$ of rooted connected pavings with $n$ darts and the set of free subgroups of index $n$ in $\Delta^+ = \mathbb{Z}_2*\mathbb{Z}_2*\mathbb{Z}_2$.
\end{Lem}

\begin{Lem}\label{lemma:pavings-conjugacy-classes}
There exists a bijection between the set $\mathcal{P}(n)$ of isomorphisms classes of connected pavings with $n$ darts and the set of conjugacy classes of free subgroups of index $n$ in $\Delta^+ = \mathbb{Z}_2*\mathbb{Z}_2*\mathbb{Z}_2$.
\end{Lem}

\section{Counting rooted pavings}\label{s:subgroups}

In this section we shall count the number of transitive triples $\langle \alpha, \beta, \gamma \rangle \subset \mathfrak{S}_n$ such that $\alpha$, $\beta$ and $\gamma$ are involutions without fixed points. Let $S_2$ be the species of such fixed-point-free involutions in $\mathfrak{S}_n$. Then since pavings correspond to triples of such involutions, for the species $P^{*}$ of labelled pavings (not necessarily connected) on $n$ darts we have
\begin{equation}\label{eqn:species-1}
P^{*} = S_2 \times S_2 \times S_2,
\end{equation}
while the species $P$ of labelled connected pavings on $n$ darts is related to $P^{*}$ by the Hurwitz equation
\begin{equation}\label{eqn:species-2}
P^{*} = E(P),
\end{equation}
where $E$ represents the species of sets.
The species $P^\circ$ of rooted connected pavings on $n$ darts can be expressed in terms of the derivative of $P$ as
\begin{equation}\label{eqn:species-3}
P^\circ = Z\cdot P^\prime,
\end{equation}
where $Z$ is the singleton species with exponential generating function $Z(z) = z$.

The above relations between species can be translated into relations between the corresponding exponential and ordinary generating functions. 

Since the generating function for $E$ is $\exp(z)$ and the direct product of species translates into the Hadamard product of series, the exponential generating functions for $S_2$, $P^*$ and $P$ are given by

\begin{equation}\label{eqn:gf-s2}
S_2(z) = \sum^\infty_{k=0} \frac{z^{2k}}{2^k k!},
\end{equation}
\begin{equation}\label{eqn:gf-p-ast}
P^*(z) = S_2(z) \odot S_2(z) \odot S_2(z) = \sum^\infty_{k=0} \frac{((2k)!)^2}{2^{3k} (k!)^3}\, z^{2k},
\end{equation}
\begin{equation}\label{eqn:gf-p}
P(z) = \log P^*(z) = \log \left( \sum^\infty_{k=0} \frac{((2k)!)^2}{2^{3k} (k!)^3}\, z^{2k} \right).
\end{equation}

The ordinary generating function for the number of rooted connected pavings with $n$ darts coincides with $P^\circ(z)$ since the species $P^\circ$ is rigid and every root assignment corresponds to $(n-1)!$ non-isomorphic labellings of the remaining darts:
\begin{equation}\label{eqn:ogf-p-circ}
P^\circ(z) = z\,\, \frac{d}{dz} \log P^*(z) = z\,\, \frac{d}{dz} \log \left( \sum^\infty_{k=0} \frac{((2k)!)^2}{2^{3k} (k!)^3}\, z^{2k} \right).
\end{equation}

Now let us write $P^*(z) = f(2z^2)$, where $f(x) = \sum^\infty_{k=0} \frac{f_k}{k!} x^k$ and $f_k = \frac{1}{2^{4k}}\, \left( \frac{(2k)!}{k!} \right)^2$. Then
\begin{equation}\label{eqn:hypergeom-1}
\frac{f_{k+1}}{f_k} = \left( k + \frac{1}{2} \right)^2.
\end{equation}
Combining equality \eqref{eqn:hypergeom-1} with the fact that $f(0) = P^*(0) = 1$, we obtain that the function $f(x)$ is hypergeometric, can be written as
\begin{equation}\label{eqn:hypergeom-2}
f(x) = {}_{2}F_{0}\left( \begin{array}{cc}
\frac{1}{2},& \frac{1}{2}\\
\cdots
\end{array}; x \right), 
\end{equation}
 and is represented by an everywhere divergent (i.e. convergent only at $z = 0$) series. As a formal series, $f(x)$ satisfies 
\begin{equation}\label{eqn:hypergeom-3}
\vartheta f(x) = x\, \left( \vartheta + \frac{1}{2} \right)^2 \, f(x),
\end{equation}
where $\vartheta = x\, \frac{d}{dx}$. c.f. \cite[Section 16.8(ii)]{DLMF}.
From equality \eqref{eqn:ogf-p-circ} we get that 
\begin{equation}\label{eqn:w}
P^{\circ}(z) = 2x\, \frac{f'(x)}{f(x)} = 2 w(x),
\end{equation}
and by combining \eqref{eqn:hypergeom-3} and \eqref{eqn:w} we see that $w(x)$ satisfies a Riccati type equation:
\begin{equation}\label{eqn:riccati}
w'(x) = \frac{(1-x) w(x) - x w^2(x) - \frac{1}{4} x}{x^2}.
\end{equation}
By \cite[Theorem 5.2]{Klazar} the function $w(x)$ is not holonomic, and therefore neither is $P^{\circ}(z)$. 

\begin{Thm}\label{thm:rooted-pavings}
The generating series $P^\circ(z) = \sum^\infty_{n=0} pav_r(n)\, z^{n}$ for the number $pav_r(n)$ of connected oriented rooted pavings with $n$ darts is non-holonomic. Its general term $pav_r(n)$ vanishes for odd values of $n$ and its asymptotic behaviour for even values of $n$ is:
\begin{equation*}
pav_r(2k) \sim 2\, \sqrt{\frac{2}{\pi}}\, \left( \frac{2}{e} \right)^k\, k^{k + 1/2}.
\end{equation*}
\end{Thm}

\begin{proof}
The above discussion contains the proof of non-holonomy. It remains to deduce the asymptotic value of $pav_r(2k)$ as $k \rightarrow \infty$. We recall that 
\begin{equation}
pav_r(2k) = [z^{2k}]\, P^\circ(z) = [z^{2k}]\, \left( z\,\, \frac{d}{dz} \log P^*(z) \right) =  
\end{equation}
\begin{equation}
= [z^{2k}]\, \left( z\,\, \frac{d}{dz} \log f(2 z^2) \right), 
\end{equation}
where
\begin{equation}
f(x) = \sum^\infty_{k=0} \frac{1}{k!}\, \left( \frac{1}{2} \right)^2_k, 
\end{equation}
according to equality \eqref{eqn:hypergeom-2}. 

Let $f(x) = \sum^\infty_{k=0} \frac{f_k}{k!} x^k$ (necessarily with $f_0 = 1$) and let $\log f(x) = \sum^\infty_{k=1} g_k x^k$. Then by \cite[Theorem 4.1]{Drmota-Nedela} (also c.f. \cite{Bender} and \cite[Theorem~7.2]{Odlyzhko}), we get that $g_k \sim \frac{f_k}{k!}$, as $k \rightarrow \infty$. 

Thus, according to the above computation
\begin{equation}
pav_r(2k) = 2^k\cdot 2k\cdot g_k \sim \frac{2^{k+1}}{(k-1)!}\, \left( \frac{1}{2} \right)^2_k.
\end{equation}

Recalling the asymptotic behaviour of the Pocchammer symbol $(a)_k$ from \eqref{Pocchammer} and Stirling's asymptotic formula $k! \sim \sqrt{2\pi k}\, e^{-k}\, k^k$ we obtain the desired asymptotic expression for $pav_r(2k)$ as $k\rightarrow \infty$. 
\end{proof}

More general asymptotic formulas for subgroup growth in free products of finite groups are given in \cite{Mul91, Mul96}, which imply the asymptotic formula in our case. 

\begin{Ex}\label{ex:rooted-pavings}
Since the generating series $P^\circ(z)$ (up to a multiple of $2$) satisfies the Riccati equation \eqref{eqn:riccati}, we obtain a recurrence relation by substituting $P^\circ(z) = \sum^\infty_{n=0} pav_r(n)\, z^{n}$ in it and equating the general term to zero:
\begin{equation}\label{eqn:recurrence42}
pav_{2n+2} = 2 (n+1)\,pav_n + \sum^{n}_{i=0} pav_{2i}\, pav_{2n-2i}, \mbox{ for } n\geq 2,
\end{equation}
with initial conditions $pav_0 = 0$, $pav_2 = 1$ and $pav_d = 0$ for all odd numbers $d\geq 1$. A similar relation is obtained in \cite[Formula 9]{Stothers}.

In order to perform our computations, a \texttt{SageMath} \cite{SageMath} worksheet \texttt{Monty} \cite{Monty} was created. With its help we found that $P^\circ(z) = z^2 + 4 z^4 + 25 z^6 + 208 z^8 + 2146 z^{10} + 26368 z^{12} + 375733 z^{14} + 6092032 z^{16} + 110769550 z^{18} + 2232792064 z^{20} + 49426061818 z^{22} + 1192151302144 z^{24} + \dots$. The coefficient sequence of $P^\circ(z)$ has index A005411 in the OEIS \cite{OEIS}. Moreover, \eqref{eqn:recurrence42} identifies it as the $S(2, -4, 1)$ self-convolutive sequence from \cite{Martin}.
\end{Ex}

By Lemma \ref{lemma:rooted-pavings-subgroups}, the above theorem can be reformulated in group-theoretic language:

\begin{Thm}\label{thm:subgroups}
The growth series $S_f(z) = \sum^\infty_{n=0} s_f(n)\, z^n$ for the number $s_f(n)$ of free subgroups of index $n$ in $\Delta^+ = \mathbb{Z}_2*\mathbb{Z}_2*\mathbb{Z}_2$ coincides with the series $P^\circ(z)$ from Theorem~\ref{thm:rooted-pavings}.
\end{Thm}

\section{Counting pavings up to isomorphism}\label{s:conj-classes}

In order to compute the generating series $\widetilde{P}(z) = \sum^\infty_{n=0} pav(n)\, z^n$ for the number $pav(n)$ of non-isomorphic connected pavings with $n$ darts, we shall employ again the species equations \eqref{eqn:species-1}--\eqref{eqn:species-3}, while replacing generating functions for the respective species with their cycle index series.

Let $C_2$ be the species of transpositions from $\mathfrak{S}_n$, $n\geq 1$. Its cycle index series can be easily expressed as $\mathcal{Z}_{C_2}(z_1, z_2, \dots) = \frac{1}{2} z^2_1 + \frac{1}{2} z_2$. The species $S_2$ of fixed-point-free involutions in $\mathfrak{S}_n$ can be expressed as $S_2 = E(C_2)$, since every involution is formed by a set of transpositions. It's also known that $\mathcal{Z}_E(z_1, z_2, \dots) = \exp \left( \sum^\infty_{n=1} \frac{z_n}{n} \right)$.

Therefore, by using \cite[\S 1.4, Th\'{e}or\`{e}me 2 (c)]{Bergeron-et-al}, the cycle index series for $S_2$ is
\begin{equation}\label{eqn:z-s2-1}
\mathcal{Z}_{S_2}(z_1, z_2, z_3, \dots) = \mathcal{Z}_E( \mathcal{Z}_{C_2}(z_1, z_2, \dots), \mathcal{Z}_{C_2}(z_2, z_4, \dots), \mathcal{Z}_{C_2}(z_3, z_6, \dots), \dots ) = 
\end{equation} 
\begin{equation}\label{eqn:z-s2-2}
 = \exp \left( \frac{z^2_1}{2} \right)\cdot \exp \left( \frac{z^2_2}{4} + \frac{z_2}{2} \right)\cdot \exp \left( \frac{z^2_3}{6} \right)\cdot \dots = \prod^\infty_{n=1}\, T_n(z_n),
\end{equation}
where
\begin{equation}\label{eqn:z-s2-3}
T_n(z_n) = \exp\left( \frac{z^2_n}{2n} + \frac{z_n}{n} \right) \mbox{ for even $n$, and } T_n(z_n) = \exp\left( \frac{z^2_n}{2n} \right) \mbox{ for odd $n$}.
\end{equation}

Thus the cycle index $\mathcal{Z}_{S_2}$ is separable, and the cycle index $\mathcal{Z}_{P^*}$ can be expressed as
\begin{equation}\label{eqn:z-p-ast}
\mathcal{Z}_{P^*}(z_1, z_2, \dots) = \prod^\infty_{n=1} (T_n\odot T_n\odot T_n)(z_n),
\end{equation}
given that $P^* = S_2 \times S_2 \times S_2$ by equation \eqref{eqn:species-1}.

By employing \cite[\S 1.4, Exercice 9 (c)]{Bergeron-et-al} together with equation \eqref{eqn:species-2}, we obtain the cycle index for the species of pavings:
\begin{equation}\label{eqn:z-p}
\mathcal{Z}_{P}(z_1, z_2, \dots) = \sum^\infty_{n=1}\, \frac{\mu(n)}{n}\, \log \mathcal{Z}_{P^*}(z_1, z_2, \dots).
\end{equation}

It follows from \cite[\S 1.2, Th\'{e}or\`{e}me 8 (b)]{Bergeron-et-al} and equations \eqref{eqn:z-p-ast}--\eqref{eqn:z-p} that the generating series $\widetilde{P}(z)$ is 
\begin{equation}\label{eqn:p-tilde-1}
\widetilde{P}(z) = \mathcal{Z}_{P}(z, z^2, z^3, \dots) = \sum^\infty_{n=1}\, \frac{\mu(n)}{n}\, \log \mathcal{Z}_{P^*}(z, z^2, z^3, \dots) = 
\end{equation}
\begin{equation}\label{eqn:p-tilde-2}
= \sum^\infty_{n=1}\, \frac{\mu(n)}{n}\, \sum^\infty_{k=1} \log (T_n\odot T_n\odot T_n)(z_n)|_{z_n = z^{nk}}.
\end{equation}

\begin{Thm}\label{thm:pavings}
The generating series $\widetilde{P}(z) = \sum^\infty_{n=0} pav(n)\, z^{n}$ for the number $pav(n)$ of connected oriented pavings with $n$ darts is given by formulas \eqref{eqn:p-tilde-1} - \eqref{eqn:p-tilde-2}. Its general term $pav(n)$ vanishes for odd values of $n$ and has the following asymptotic behaviour for even values of $n$:
\begin{equation*}
pav(2k) \sim \sqrt{\frac{2}{\pi}}\, \left( \frac{2}{e} \right)^k\, k^{k - 1/2}.
\end{equation*}
\end{Thm}

\begin{proof}
By an argument analogous to that of \cite[Section 7.1]{Drmota-Nedela}, we obtain $pav(2k) \sim \frac{pav_r(2k)}{2k}$ as $k\rightarrow \infty$. Now the claim follows from Theorem~\ref{thm:rooted-pavings}.
\end{proof}

\begin{Ex}\label{ex:pavings}
By using \texttt{Monty} \cite{Monty}, we computed the initial sequence of coefficients for $\widetilde{P}(z)$ and obtained that $\widetilde{P}(z) = z^2 + 4 z^4 + 11 z^6 + 60 z^8 + 318 z^{10} + 2806 z^{12} + 29359 z^{14} + 396196 z^{16} + 6231794 z^{18} + 112137138 z^{20} + \dots$.  The coefficient sequence of $\widetilde{P}(z)$ has index A002831 in the OEIS \cite{OEIS}, which represents the number of edge-$3$-coloured trivalent multi-graphs\footnote{i.e. with multiple edges} on $2n$ vertices, $n\geq 0$, without loops; let this number be $tri(n)$ and let $\widetilde{G}(z) = \sum_{n\geq 0} tri(n) z^{2n}$. Thus the number of isomorphism classes of transitive triples of fixed-point-free involutions from $\mathfrak{S}_{2n}$ equals both $pav(n)$ (as shown above) and $tri(n)$.  

Indeed, in order to create a labelled (not necessarily connected) edge-$3$-coloured trivalent multi-graph without loops, we need to choose three matchings in the set of $2n$ vertices, which we may think of as a set $V = [2n]$. Each matching will consist of edges of same colour, say red (R), green (G) or blue (B). A matching of some colour $c \in \{R, G, B\}$ is then described as a product $\sigma$ of disjoint transpositions $(i,j)$ corresponding to the two vertices $i$ and $j$ from $V$ joined by an edge. Since there are no loops, each matching has exactly $n$ edges, and $\sigma_c$ has no fixed points. See \cite{Read} for a general approach to enumeration of graphs with ``local restrictions''. 

Let $G^*$ be the species of vertex-labelled edge-$3$-coloured trivalent multigraph without loops, and let $G$ be its connected counterpart. Then $G^*$ can be described as a species of triples of  fixed-point-free involutions $\langle \sigma_R, \sigma_G, \sigma_B \rangle$, so $G^* \cong P^*$ and, subsequently, $G \cong P$, as species. From this isomorphism, we get that, in particular, $\widetilde{G}(z) = \widetilde{P}(z)$ and the coefficient sequence of $\widetilde{P}(z)$ coincides with A002831.
\end{Ex}

\begin{Thm}\label{thm:conjugacy-classes}
The growth series $C_f(z) = \sum^\infty_{n=0} c_f(n)\, z^n$ for the number $c_f(n)$ of conjugacy classes of free subgroups of index $n$ in $\Delta^+ = \mathbb{Z}_2*\mathbb{Z}_2*\mathbb{Z}_2$ coincides with the series $\widetilde{P}(z)$ from Theorem~\ref{thm:pavings}.
\end{Thm}

\begin{Ex}\label{ex:conjugacy-classes}
Below we present the non-isomorphic pavings with $n\leq 4$ darts, which also provide a classification for all conjugacy classes of free subgroups of index $\leq 4$ in $\Delta^+$ in view of Lemma \ref{lemma:pavings-conjugacy-classes} and the preceding discussion. The corresponding pavings can easily be classified by hand.

The conjugacy growth series for $\Delta^+$ is given in Example~\ref{ex:pavings}. An independent computation with GAP \cite{GAP} by issuing \texttt{LowIndexSubgroupsFPGroup} command gives matching results. We may also use \texttt{FactorCosetAction} command to observe the action of a conjugacy class representative on its cosets. 

\begin{figure}[h]
\includegraphics[scale=.42]{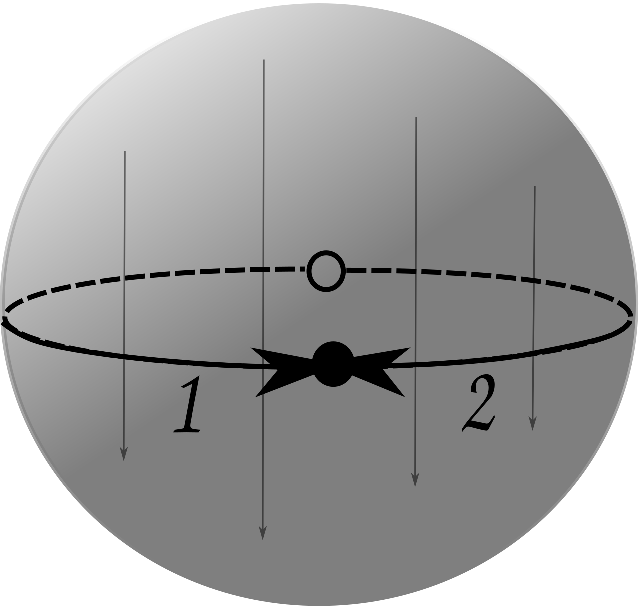}
\caption{Paving $P_1$ with $2$ darts produced by face-glueing. The face identification $(x,y,z) \mapsto (x,y,-z)$ is depicted by arrows.}\label{fig:pav-2-1}
\end{figure}

Let $P = \langle D; \alpha, \beta, \gamma\rangle$ be a paving. For the case of two darts $D = \{ 1, 2 \}$ we obtain only one paving $P_1$ with
\begin{equation}\label{eqn:pav-2-1}
( \alpha, \beta, \gamma ) \longmapsto ( (1,2), (1,2), (1,2) ).
\end{equation}

\begin{figure}[h]
\includegraphics[scale=.42]{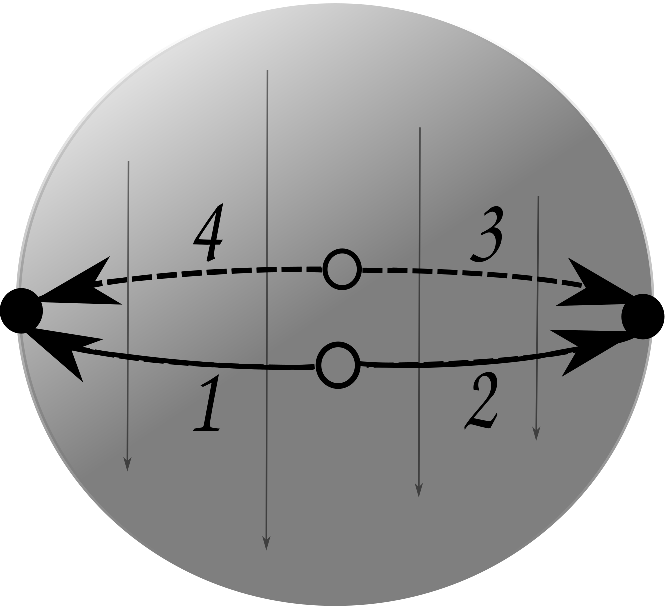}
\caption{Paving $P_2$ with $4$ darts produced by face-glueing. The face identification $(x,y,z) \mapsto (x,y,-z)$ is depicted by arrows.}\label{fig:pav-4-1}
\end{figure}

This paving is glued from a single $3$-ball $B_1$ with a map $H_1$ on it, as shown in Figure~\ref{fig:pav-2-1}. If we suppose that $B_1$ is a unit ball centred at the origin of $\mathbb{R}^3$, then the identification of the faces of $H_1$ can be described by the transformation  $(x,y,z) \mapsto (x,y,-z)$. This paving has $\mathrm{f}$-vector $(1,1,1,1)$.

For the case of four darts, that is, $D = \{ 1, 2, 3, 4 \}$, we get four more pavings.

The first one is $P_2$ with
\begin{equation}\label{eqn:pav-4-1}
( \alpha, \beta, \gamma ) \longmapsto ( (1,2)(3,4), (1,2)(3,4), (1,3)(2,4) ).
\end{equation}

Here, $P_2$ is topologically represented by glueing the boundary of a $3$-ball $B_2$ with a map $H_2$ on it, as depicted in Figure~\ref{fig:pav-4-1}. Again, such a glueing can be described by the transformation $(x,y,z) \mapsto (x,y,-z)$. This paving has  $\mathrm{f}$-vector $(2,2,1,1)$

The next paving $P_3$ has
\begin{equation}\label{eqn:pav-4-2}
( \alpha, \beta, \gamma ) \longmapsto ( (1,2)(3,4), (1,3)(2,4), (1,2)(3,4) ).
\end{equation}

\begin{figure}[h]
\includegraphics[scale=.42]{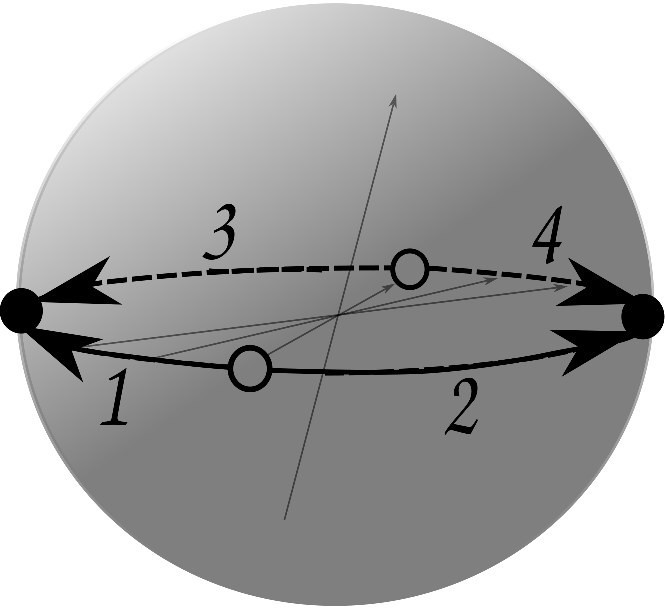}
\caption{Paving $P_3$ with $4$ darts produced by face-glueing. The face identification $(x,y,z) \mapsto (-x,-y,-z)$ is depicted by arrows.}\label{fig:pav-4-2}
\end{figure}

It is depicted in Figure~\ref{fig:pav-4-2}, and topologically is a single $3$-ball $B_3$ with a map $H_3$ on it, whose faces are identified accordingly. The glueing transformation in this case can be described as $(x,y,z) \mapsto (-x,-y,-z)$.  This paving has $\mathrm{f}$-vector $(1,1,1,1)$.

An easy computation yields that each of $P_i$, $i=1,2,3$, has Euler characteristic $\chi(P_i) = 0$, as any three-dimensional manifold \cite[Theorem~4.3]{FM}, and it can be readily seen that $P_1$ and $P_2$ are homeomorphic to the three-sphere $\mathbb{S}^3$, while $P_3$ is homeomorphic to the real projective space $\mathbb{R}P^3$.

As for the remaining two pavings $P_4$ and $P_5$, both of them correspond topologically to glueing two disjoint balls along their boundaries, and the Euler characteristic for both is $0$; thus each is a manifold by \cite[Theorem~4.3]{FM}. Moreover, each is an orientable manifold of Heegaard genus zero, and thus again homeomorphic to $\mathbb{S}^3$ \cite[Ch. 5, \S 1]{FM}.

\begin{figure}[h]
\includegraphics[scale=.42]{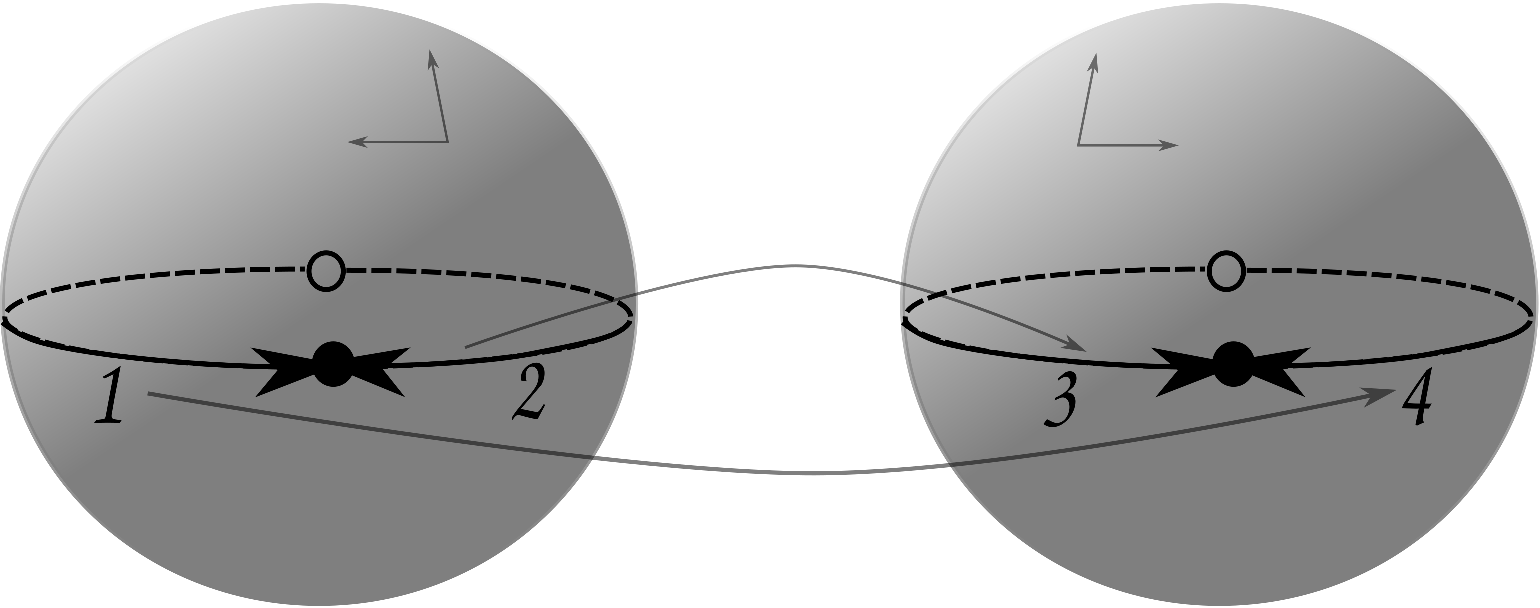}
\caption{Paving $P_4$ with $4$ darts produced by face-glueing. The face identification is depicted by arrows.}\label{fig:pav-4-3}
\end{figure}

For $P_4$ we have
\begin{equation}\label{eqn:pav-4-3}
( \alpha, \beta, \gamma ) \longmapsto ( (1,2)(3,4), (1,3)(2,4), (1,3)(2,4) ), 
\end{equation}
which is a combinatorial description for the two $3$-balls $B_{4,1}$ and $B_{4,2}$ shown in Figure~\ref{fig:pav-4-3}, each with a connected map $H_{4,1}$, respectively $H_{4,2}$, on it. The faces of those maps are identified by an orientation-reversing transformation on $\partial B_{4,1} \cong \mathbb{S}^2 \cong \partial B_{4,2}$. This paving has $\mathrm{f}$-vector $(1,1,2,2)$.

\begin{figure}[h]
\includegraphics[scale=.42]{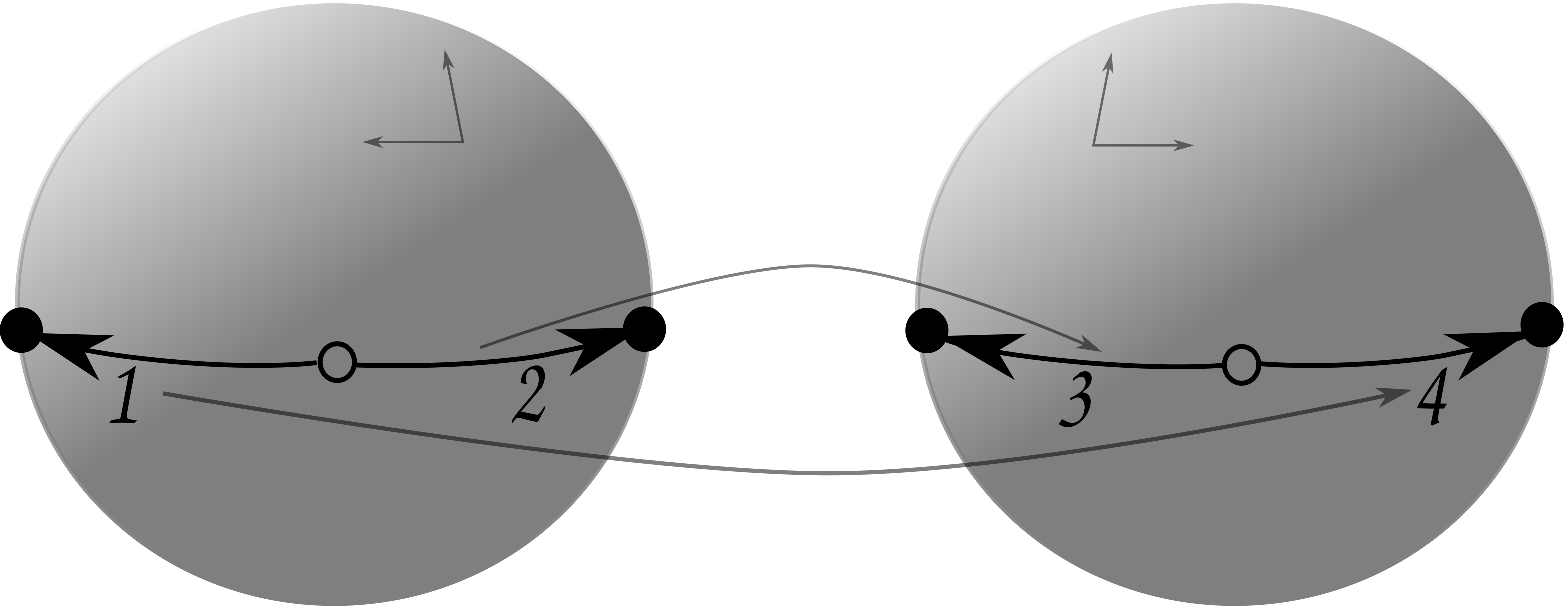}
\caption{Paving $P_5$ with $4$ darts produced by face-glueing. The face identification is depicted by arrows.}\label{fig:pav-4-4}
\end{figure}

Finally, for $P_5$ we obtain
\begin{equation}\label{eqn:pav-4-4}
( \alpha, \beta, \gamma ) \longmapsto ( (1,2)(3,4), (1,3)(2,4), (1,4)(2,3) ). 
\end{equation}
In this case two $3$-balls $B_{5,1}$ and $B_{5,2}$ shown in Figure~\ref{fig:pav-4-4} are identified along their boundaries. The identification is described by the glueing of the faces of the corresponding maps $H_{5,1}$ and $H_{5,2}$ on their boundaries. The $\mathrm{f}$-vector of this paving is $(2,1,1,2)$.
\end{Ex}

\section{Counting pavings of the three-sphere}\label{sec:counting}

Let us consider a Heegaard splitting $H \cup H' = \mathbb{S}^3$ of the three-sphere $\mathbb{S}^3$, where the handlebodies $H$ and $H'$ are glued along their common boundary $\Sigma = H \cap H'$. If we suppose that $\Sigma$ has a map on it, then such a splitting $H \cup H'$ turns into a paving. Indeed, we can split each edge on $\Sigma$ into two darts, and then double each dart, such that we have two maps $\Sigma$ and $\Sigma'$ corresponding to the boundaries of $H$ and $H'$; then we can write down the permutation representation for each of them. Finally we write down a permutation that pairs the darts of $\Sigma$ with those of $\Sigma'$: whichever map we choose for $\Sigma$ will determine the map on $\Sigma'$.

We can also think of $\mathbb{S}^3$ as $\mathbb{E}^3 \cup \infty$ and then delete from $\mathbb{E}^3$ a genus $g$ handlebody $H$. Then the closure $H'$ of the complement $\mathbb{S}^3 \setminus H$ will be a genus $g$ handlebody $H'$, and the surfaces of $H$ and $H'$ will have opposite orientations. 
Thus, if we choose a map $\Sigma$ on a genus $g$ surface of a handlebody $H$, we automatically imprint its chiral (i.e. having inverse orientation) counterpart $\Sigma'$ on the surface of $H'$, so we created a paving $P$ with underlying map $H_P = \Sigma \sqcup \Sigma'$. 

If two pavings are isomorphic, then their underlying maps are necessarily isomorphic. By the above construction, we have at least as many non-isomorphic oriented pavings $P$ on $2n$ darts representing $\mathbb{S}^3$ as the total number of non-isomorphic oriented maps $H$ on $n$ darts. Thus, the number of pavings representing $\mathbb{S}^3$ grows super-exponentially with respect to $n$. 

We remark that the complexity of our paving $P$ can be easily computed. If $\mathrm{f}(P) = (f_0, f_1, f_2, f_3)$ then $\chi(H) = f_2 - f_1 + f_0 = 2 - 2g$, where $g$ is the genus of the surface carrying the map $H$, and $f_3 = 2$. Thus $\mathrm{c}(P) = f_3 - f_2 + f_1 - f_0 = 2 - (2 - 2g) = 2g$, and its value will vary over the set of maps on $n$ darts. This fact motivates the following questions.

\begin{Que}
Let $\mathcal{P}_c(n)$ be the set of pavings with $n$ darts, all of fixed complexity $c$ (although not necessarily of a fixed homeomorphism type). Is it true that $\mathrm{card}\, \mathcal{P}_c(n) \sim C_1 \, \exp(C_2 n)$ for some $C_1$, $C_2 > 0$, if $n$ is great enough?
\end{Que}   

\begin{Que}
Let $\mathcal{P}_M(n)$ be the set of pavings with $n$ darts, all homeomorphic to a given manifold (or a cell complex) $M$. Is it true that $\mathrm{card}\, \mathcal{P}_M(n) \sim C_1 \, \exp(C_2 n)$ for some $C_1$, $C_2 > 0$, if $n$ is great enough?
\end{Que} 

\begin{table}[ht]
\begin{tabular}{|l|l|l|l|l|}
\hline
Darts & Maps   & Manifolds & With boundary & Genera of boundary components                                                                                                                                                                                                                                                                                                                                   \\ \hline
2     & 1      & 1         &               &                                                                                                                                                                                                                                                                                                                                                                 \\ \hline
4     & 4      & 4         &               &                                                                                                                                                                                                                                                                                                                                                                 \\ \hline
6     & 11     & 10        & 1             & {[}1{]} x 1                                                                                                                                                                                                                                                                                                                                                     \\ \hline
8     & 60     & 44        & 16            & \begin{tabular}[c]{@{}l@{}}{[}1{]} x 15 \\ {[}1,1{]} x 1\end{tabular}                                                                                                                                                                                                                                                             \\ \hline
10    & 318    & 183       & 135           & \begin{tabular}[c]{@{}l@{}}{[}1{]} x 108 \\ {[}2{]} x 27\end{tabular}                                                                                                                                                                                                                                                             \\ \hline
12    & 2806   & 1141      & 1665          & \begin{tabular}[c]{@{}l@{}}{[}1{]} x 1168 \\ {[}1,1{]} x 30 \\ {[}2{]} x 467\end{tabular}                                                                                                                                                                                                           \\ \hline
14    & 29359  & 7992      & 21367         & \begin{tabular}[c]{@{}l@{}}{[}1{]} x 11886 \\ {[}1,1{]} x 57 \\ {[}2{]} x 8161 \\ {[}3{]} x 1263\end{tabular}                                                                                                                                                         \\ \hline
16    & 396196 & 66616     & 329580        & \begin{tabular}[c]{@{}l@{}}{[}1{]} x 140507 \\ {[}1,1{]} x 2149 \\ {[}1,1,1{]} x 4 \\ {[}1,2{]} x 247 \\ {[}2{]} x 136833 \\ {[}2,2{]} x 40 \\ {[}3{]} x 49800\end{tabular} \\ \hline
\end{tabular}
\caption{Statistical data on maps with $n \leq 16$ darts: for each $n$ the number of manifold pavings and manifolds with boundary (pavings with non-spherical links) is given. For manifolds with boundary, $[g_1, g_2, \ldots, g_k]\, \text{x}\, m$ indicates that there are $m$ pavings giving rise to a manifold with boundary components of genera $g_1, g_2, \ldots, g_k$. The pavings are classified up to combinatorial isomorphism, and \textit{not} up to homeomorphism of the respective manifolds.}\label{tab:stats}
\end{table}

\section{Statistics for pavings on $n \leq 16$ darts}\label{sec:stats}

We created a \texttt{Rust} code called \texttt{Nem} \cite{Nem} in order to perform recursive enumeration of all pairwise non-isomorphic Schreier graphs of index $\leq 16$ free subgroups in $\Delta^+ = \mathbb{Z}_2*\mathbb{Z}_2*\mathbb{Z}_2$.   

Since Schreier graphs of index $2n$ free subgroups $H < \Delta^+$ have relatively easy combinatorial structure (they are rooted edge-$3$-coloured trivalent multi-graphs on $2n$ vertices), and their (unrooted) isomorphism can be verified in at most $O(n^2)$ time by an easy trial-and-error algorithm, our approach is more efficient than a direct attempt at classifying pavings as $3$-dimensional objects. 

Some statistical information produced by \texttt{Nem} for pavings on $n\leq 16$ darts is available in Table~\ref{tab:stats}, and the complete output of \texttt{Nem} can be downloaded together with its source code \cite{Nem}. For our computation we used the ``Cervino'' computational cluster with $64$ CPU cores and  $126$ Gb of RAM courtesy of the University of Neuch\^atel.

\vspace*{0.15in}

\begin{table}[!hb]
\centering
\begin{tabular}{lll}
\begin{tabular}[c]{@{}l@{}}\it R\'emi Bottinelli \\ \it Institut de Math\'{e}matiques\\ \it Universit\'{e} de Neuch\^{a}tel\\ \it Rue Emile-Argand 11\\ \it CH-2000 Neuch\^{a}tel\\ \it Suisse/Switzerland\\ \it remi.bottinelli@unine.ch\end{tabular} & \begin{tabular}[c]{@{}l@{}} \it Laura Ciobanu\\ \it School of Mathematical\\ \it and Computer Sciences\\ \it Heriot-Watt University\\ \it 6100 Main Street\\ \it Edinburgh EH14 4AS, UK\\ \it l.ciobanu@hw.ac.uk\end{tabular} & \begin{tabular}[c]{@{}l@{}}\it Alexander Kolpakov\\ \it Institut de Math\'{e}matiques\\ \it Universit\'{e} de Neuch\^{a}tel\\ \it Rue Emile-Argand 11\\ \it CH-2000 Neuch\^{a}tel\\ \it Suisse/Switzerland\\ \it kolpakov.alexander@gmail.com\end{tabular}
\end{tabular}
\end{table}


\begin{thebibliography}{99}

\bibitem{AB}{D. Arqu\`{e}s, J.-F. B\'{e}raud}, {``Rooted maps on orientable surfaces, Riccati's equation and continued fractions''}, Discrete Math. \textbf{215}, 1--12 (2000). 

\bibitem{AK}{D. Arqu\`es, P. Koch},{``Pavages tridimensionels''}, Bigre+Globule, \textbf{61}--\textbf{62}, 5--15 (1989).

\bibitem{BPR}{H. Baik, B. Petri, J. Raimbault}, {``Subgroup growth of virtually cyclic right-angled Coxeter groups and their free products''}, Combinatorica \textbf{39}, 779--811 (2019).

\bibitem{Bender}{E.~A. Bender}, {``An asymptotic expansion for the coefficients of some formal power series''}, J. London Math. Soc. \textbf{9}, 451--458 (1975).

\bibitem{Bergeron-et-al}{F. Bergeron, G. Labelle, P. Leroux}, {``Th\'{e}orie des esp\`{e}ces et combinatoire des structures arborescentes''}, Publications du LaCIM, Universit\'{e} du Qu\'{e}bec \`{a} Montr\'{e}al, 1994.

\bibitem{Monty}{L. Ciobanu, A. Kolpakov}, {``Monty: a SageMath worksheet''}, GitHub:\, \url{https://github.com/sashakolpakov/monty-3d}

\bibitem{Nem}{R. Bottinelli}, {``Nem: Na\"ive Enumeration of Maps''}, GitHub:\, \url{https://github.com/bottine/nem}

\bibitem{Breda-Mednykh-Nedela}{A. Breda, A. Mednykh, R. Nedela}, {``Enumeration of maps regardless of genus. Geometric approach''}, Discrete Math. \textbf{310}, 1184--1203 (2010).

\bibitem{CK-2017}{L. Ciobanu, A. Kolpakov}, {``Free subgroups of free products and combinatorial hypermaps''}, Discrete Math. \textbf{342}, 1415 -- 1433 (2019).

\bibitem{DLMF}{\texttt{DLMF}}, {Digital Library of Mathematical Functions}, \url{http://dlmf.nist.gov/}

\bibitem{Drmota-Nedela}{M. Drmota, R. Nedela}, {``Asymptotic enumeration of reversible maps regardless of genus''}, Ars Mathematica Contemporanea \textbf{5} (1), 77--97 (2012).

\bibitem{Flajolet-Sedgewick}{P. Flajolet, R. Sedgewick}, {``Analytic combinatorics''}, Cambridge University Press: Cambridge, 2009.

\bibitem{FM}{A. T. Fomenko, S. V. Matveev}, {``Algorithmic and Computer Methods for Three-Manifolds''}, Kluwer: Dordrecht, 1997.

\bibitem{GAP}{\texttt{GAP}}, {GAP - Groups, Algorithms, Programming}, \url{http://www.gap-system.org/}

\bibitem{Gurau}{R. Gurau}, {``Random tensors''}, Oxford University Press: Oxford, New York, 2017. 

\bibitem{Jones-Singerman}{G.~A. Jones, D. Singerman}, {``Theory of maps on orientable surfaces''}, Proc. London Math. Soc. \textbf{37} (3), 273--307  (1978).

\bibitem{Joyal}{A. Joyal}, {``Une th\'{e}orie combinatoire des s\'{e}ries formelles''}, Adv. Math. \textbf{42} (1), 1--82 (1981).

\bibitem{Klazar}{M. Klazar}, {``Irreducible and connected permutations''}, IUUK-CE-ITI pre-print series (2003); available \href{http://iti.mff.cuni.cz/series/2003.html}{on-line}

\bibitem{Lienhardt}{P. Lienhardt}, {``Topological models for boundary representation: a comparison with $n$--dimensional generalized maps''}, Computer-Aided Design \textbf{23} (1), 59--82 (1991).

\bibitem{LS} {A. Lubotzky, D. Segal}, {``Subgroup Growth''}, Progress in Mathematics \textbf{212}, Birkh\"{a}user Verlag: Basel, 2003.

\bibitem{Martin} {R.J. Martin, M.J. Kearney}, {``An exactly solvable self-convolutive recurrence''}, Aequationes Math. \textbf{80} (3), 291--318 (2010).

\bibitem{M}{A. Mednykh}, {``Enumeration of unrooted hypermaps''}, Electron. Notes Discrete Math. \textbf{28},  207--214 (2007).

\bibitem{M-j-algebra}{A. Mednykh}, {``Counting conjugacy classes of subgroups in a finitely generated group''}, J. Algebra \textbf{320} (6), 2209--2217 (2008).

\bibitem{MN1}{A. Mednykh, R. Nedela}, {``Enumeration of unrooted maps of a given genus''}, J.~Combin. Theory, Ser.~B \textbf{96} (5), 706--729 (2006).

\bibitem{MN2}{A. Mednykh, R. Nedela}, {``Enumeration of unrooted hypermaps of a given genus''}, Discrete Math. \textbf{310} (3), 518--526 (2010).

\bibitem{Mul91}{T.~W. M\"{u}ller}, {``Combinatorial aspects of finitely generated virtually free groups''}, J. London Math. Soc. \textbf{44} (2), 75--94 (1991). 

\bibitem{Mul96}{T.~W. M\"{u}ller}, {``Subgroup growth of free products''}, Invent. Math. \textbf{126} (1),  111--131 (1996).

\bibitem{MSP1}{T.~W. M\"{u}ller, J.-C. Schlage-Puchta}, {``Classification and statistics of finite index subgroups in free products''}, Adv. Math. \textbf{188} (1), 1--50 (2004).

\bibitem{MSP2}{T.~W. M\"{u}ller, J.-C. Schlage-Puchta}, {``Character theory of symmetric groups, subgroup growth of Fuchsian groups, and random walks''}, Adv. Math. \textbf{213} (2), 919--982 (2007).

\bibitem{MSP3}{T.~W. M\"{u}ller, J.-C. Schlage-Puchta}, {``Statistics of isomorphism types in free products''}, Adv. Math. \textbf{224} (2), 707--730 (2010). 

\bibitem{Odlyzhko}{A.~M. Odlyzko}, {``Asymptotic enumeration methods''}, in: R. L. Graham, M. Gr\"{o}tschel, and L. Lov\'{a}sz (eds.), Handbook of Combinatorics, vol. 2, Elsevier, 1995, 1063--1229.

\bibitem{Petitot-Vidal}{M. Petitot, S. Vidal}, {``Counting Rooted and Unrooted Triangular Maps''}, J. Nonlin. Systems Applications \textbf{1} (2), 51--57 (2010).

\bibitem{Zeilberger-et-al}{M. Petkovsek, H.~S. Wilf, D. Zeilberger}, {``A = B (with foreword by D.E. Knuth)''}, A.~K. Peters: Wellesley MA, 1996.

\bibitem{Ratcliffe}{J. G. Ratcliffe}, {``Foundations of Hyperbolic Manifolds''}, 2nd Edition, Springer: Graduate Texts in Mathematics \textbf{149}

\bibitem{Read}{R.~C. Read}, {``The Enumeration of Locally Restricted Graphs (I)''}, J. London Math. Soc. \textbf{34} (4), 417--436 (1959).

\bibitem{SageMath}{\texttt{SageMath}}, {the Sage Mathematics Software System (version 9.1)}, \url{https://www.sagemath.org}

\bibitem{OEIS}{N.~J.~A. Sloane et al.}, {``The On-line Encyclopaedia of Integer Sequences''}, \texttt{http://oeis.org}

\bibitem{Spehner}{J.-C. Spehner}, {``Merging in maps and in pavings''}, Theoret. Comput. Sci. \textbf{86}, 205--232 (1991).

\bibitem{Stothers-modular}{W.W. Stothers}, {``The number of subgroups of given index in the modular group''}, Proc.~Roy.~Soc.~Edinburgh, \textbf{78a}, 105--112 (1978).

\bibitem{Stothers}{W.W. Stothers}, {``Free Subgroups of Free Products of Cyclic Groups''}, Math. Comput. \textbf{32} (144), 1274--1280 (1978).

\bibitem{Thurston-notes}{W.~P. Thurston}, {``The geometry and topology of three-manifolds''}, Princeton lecture notes (1978--1981).

\bibitem{Vidal}{S. Vidal}, {``Sur la classification et le denombrement des sous-groupes du groupe modulaire et de leurs classes de conjugaison''}, \texttt{arXiv:math/0702223}. 

\end{thebibliography}
\end{document}